\documentclass[12pt]{amsart}

\usepackage{mathrsfs, amssymb,amsthm, amsfonts, amsmath}
\usepackage[all]{xy}
\usepackage{upgreek}
\usepackage{amsmath}

\usepackage{tabularx,multirow,makecell}
\newcolumntype{Y}{>{\centering\arraybackslash}X}
\usepackage{setspace}
\newtheorem{theorem}[equation]{Theorem}
\newtheorem*{theorem*}{Theorem}
\newtheorem{lemma}[equation]{Lemma}
\newtheorem{corollary}[equation]{Corollary}

\newtheorem{proposition}[equation]{Proposition}

\theoremstyle{definition}

\newtheorem{definition}[equation]{Definition}
\newtheorem*{definition*}{Definition}

\theoremstyle{remark}
\newtheorem{remark}[equation]{Remark}

\makeatletter\@addtoreset{equation}{section}
\makeatother

\usepackage{textcomp}
\usepackage{hyperref}

\newcommand{\CC}{\mathbb{C}}

\newcommand{\QQ}{\mathbb{Q}}

\newcommand{\ZZ}{\mathbb{Z}}
\newcommand{\PP}{\mathbb{P}}

\newcommand{\OOO}{{\mathscr{O}}}

\newcommand{\KKK}{{\mathscr{K}}}

\newcommand{\chit}{\chi_{\mathrm{top}}}
\newcommand{\ad}{\mathrm{a}}
\newcommand{\cc}{\mathrm{c}}
\newcommand{\bb}{\mathrm{b}}
\newcommand{\hh}{\mathrm{h}}

\newcommand{\Aut}{\operatorname{Aut}}
\newcommand{\GL}{\operatorname{GL}}

\newcommand{\Bir}{\operatorname{Bir}}
\newcommand{\PGL}{\operatorname{PGL}}

\newcommand{\Fix}{\operatorname{Fix}}

\newcommand{\rk}{\operatorname{rk}}

\newcommand{\red}{\operatorname{red}}

\def \ge {\geqslant}
\def \le {\leqslant}

\date{}

\title{Automorphism groups of compact complex surfaces}

\author{Yuri Prokhorov}
\author{Constantin Shramov}

\address{
Steklov Mathematical Institute of Russian Academy of Sciences, 8 Gubkina st.,
Moscow, 119991, Russia
\newline
National Research University Higher School of Economics, Laboratory of Algebraic Geometry, 6 Usacheva str., Moscow, 119048, Russia
}

\email{prokhoro@mi.ras.ru}
\email{costya.shramov@gmail.com}

\thanks{This work is supported by the Russian Science Foundation under grant \textnumero 18-11-00121.}

\begin{document}

\begin{abstract}
We study automorphism groups and birational automorphism groups of compact complex surfaces.
We show that the automorphism group of such a surface $X$ is always Jordan, and
the birational automorphism group is Jordan unless $X$ is birational to a product of an elliptic and a rational curve.
\end{abstract}

\maketitle
\tableofcontents

\section{Introduction}

It often happens that some infinite subgroups exhibit a nice and simple behavior on the level
of their finite subgroups. An amazing example of such a situation is given by
the following result due to
C.\,Jordan (see~\cite[Theorem~36.13]{Curtis-Reiner-1962}).

\begin{theorem}
\label{theorem:Jordan}
There is a constant~\mbox{$J=J(n)$} such that for every
finite subgroup~\mbox{$G\subset\GL_n(\CC)$}
there exists
a normal abelian subgroup~\mbox{$A\subset G$} of index at most $J$.
\end{theorem}

This motivates the following definition.

\begin{definition}[{see \cite[Definition~2.1]{Popov2011}}]
\label{definition:Jordan}
A group~$\Gamma$ is called \emph{Jordan}
(alternatively, we say
that~$\Gamma$ \emph{has Jordan property})
if there is a constant~$J$ such that
for every finite subgroup~\mbox{$G\subset\Gamma$} there exists
a normal abelian subgroup $A\subset G$ of index at most~$J$.
\end{definition}

In other words, Theorem~\ref{theorem:Jordan} says that the group $\GL_n(\CC)$ is
Jordan. The same applies to any linear algebraic group, since it can be realized
as a subgroup
of a general linear group.

It was noticed by J.-P.\,Serre that Jordan property sometimes holds
for groups of birational automorphisms.

\begin{theorem}[{\cite[Theorem~5.3]{Serre2009},
\cite[Th\'eor\`eme~3.1]{Serre-2008-2009}}]
\label{theorem:Serre}
The group of birational automorphisms of $\PP^2$
over the field $\CC$ (or any other field of characteristic $0$) is Jordan.
\end{theorem}

Yu.\,Zarhin pointed out in~\cite{Zarhin10} that there are projective complex
surfaces
whose birational automorphism groups are not Jordan; they are birational to
products $E\times\PP^1$, where $E$ is an elliptic curve.
The following result of V.\,Popov classifies projective
surfaces with non-Jordan birational automorphism
groups.

\begin{theorem}[{\cite[Theorem~2.32]{Popov2011}}]
\label{theorem:Popov}
Let $X$ be a projective surface over $\CC$.
Then the group of birational automorphisms of $X$ is not Jordan if and only if~$X$
is birational to~\mbox{$E\times\PP^1$}, where $E$ is an
elliptic curve.
\end{theorem}

Automorphism groups having Jordan property were studied recently in many
different contexts.
Yu.\,Prokhorov and C.\,Shramov in \cite[Theorem~1.8]{ProkhorovShramov-RC} and
\cite[Theorem~1.8]{Prokhorov-Shramov-2013} proved that this property holds for
groups of birational selfmaps of rationally connected algebraic varieties,
and some other algebraic varieties of arbitrary dimension. Actually, their
results
were initially obtained modulo a conjectural boundedness of terminal Fano
varieties
(see e.\,g.~\cite[Conjecture~1.7]{ProkhorovShramov-RC}), which
was recently proved by C.\,Birkar in~\cite[Theorem~1.1]{Birkar}.
Also Yu.\,Prokhorov and C.\,Shramov classified Jordan birational automorphism
groups
of algebraic threefolds in~\cite{ProkhorovShramov-dim3}. Some results about
birational
automorphisms of conic bundles were obtained by T.\,Bandman and Yu.\,Zarhin
in~\cite{BandmanZarhin2015a}. For other results on
Jordan birational automorphism groups see \cite{Prokhorov-Shramov-JCr3},
\cite{Prokhorov-Shramov-p-groups}, and \cite{Yasinsky2016a}.

S.\,Meng and D.-Q.\,Zhang proved in \cite{MengZhang} that the automorphism group
of any projective variety is Jordan, and J.\,H.\,Kim generalized this
to automorphism groups of compact K\"ahler manifolds
in~\cite{Kim}.
T.\,Bandman and Yu.\,Zarhin proved
a similar result for automorphism groups of quasi-projective surfaces
in~\cite{BandmanZarhin2015},
and also in some particular cases in arbitrary dimension in~\cite{BandmanZarhin2017}.
For a survey of some other relevant results see~\cite{Popov-Jordan}.

\'E.\,Ghys asked (following a more particular question posed earlier by W.\,Feit)
whether the diffeomorphism group
of a smooth compact manifold is always Jordan. Recently
B.\,Csik\'os, L.\,Pyber, and E.\,Szab\'o in
\cite{CsikosPyberSzabo} provided a counterexample
following the method of~\cite{Zarhin10};
see also \cite{Riera-LieGroups} for a further development of this
method, and \cite[Corollary~2]{Popov-Diff}
for a non-compact counterexample.
However, Jordan property holds for diffeomorphism groups in many cases;
see \cite{Riera2016}, \cite{Riera-Spheres}, \cite{TurullRiera2015},
\cite{Riera-OddCohomology},
\cite{GuazziZimmermann},
\cite{Zimmermann-Survey},
\cite{Zimmermann-ConnectedSum},
\cite{Zimmermann2014},
\cite{Zimmermann-Isometries},
and references therein.
Also there are results for groups of symplectomorphisms, see
\cite{Riera-Symp} and \cite{Riera-HamSymp}.

The goal of this paper is to
generalize Theorem~\ref{theorem:Popov},
and to some extent the results of~\cite{MengZhang} and~\cite{Kim},
to a different setting, namely, to the
case of
compact complex surfaces (see~\S\ref{section:minimal-surfaces} below for basic
definitions and background).
There are some particular cases that are already known.
For instance, automorphism groups of Inoue surfaces (see \cite{Inoue1974})
and primary Kodaira surfaces (see \cite[\S6]{Kodaira-structure-1},
\cite[\S\,V.5]{BHPV-2004}) were studied in~\cite{ProkhorovShramov-IK}.

\begin{theorem}[{\cite[Theorem~1.2]{ProkhorovShramov-IK}}]
\label{theorem:IK}
Let $X$ be either an Inoue surface or a primary Kodaira surface.
Then the automorphism group of $X$ is Jordan.
\end{theorem}

We prove the following.

\begin{theorem}\label{theorem:Aut}
Let $X$ be a connected compact complex surface.
Then the automorphism group of $X$ is Jordan.
\end{theorem}

One can also show (see~\cite[Theorem~1.3]{Riera-OddCohomology} or Theorem~\ref{theorem:Riera-generators} below)
that the number of generators of any finite subgroup of the automorphism group of a compact complex surface~$X$,
and actually of any finite subgroup of the diffeomorphism group of an arbitrary compact manifold, is bounded by a constant
that depends only on~$X$.

The main result of this paper is as follows.

\begin{theorem}
\label{theorem:main}
Let $X$ be a connected compact complex surface.
Then the group of birational automorphisms of $X$ is not Jordan if and only if~$X$
is birational to~\mbox{$E\times\PP^1$}, where~$E$ is an
elliptic curve. Moreover, there always exists
a constant~\mbox{$R=R(X)$} such that every finite subgroup of the birational automorphism group of~$X$
is generated by at most~$R$ elements.
\end{theorem}

The plan of the paper is as follows.
In \S\ref{section:Jordan} we collect some elementary facts about Jordan
property, and other boundedness
properties for subgroups.
In \S\ref{section:minimal-surfaces} we recall the basic facts from the theory of
compact complex
surfaces, most importantly their Enriques--Kodaira classification.
In \S\ref{section:automorphisms} we recall some important general facts
concerning automorphisms
of complex spaces.
In \S\ref{section:class-VII} we study automorphism groups
of non-projective surfaces with non-zero topological Euler characteristic;
an important subclass of such surfaces is formed by minimal surfaces of class~VII
with non-zero second Betti number (which are still not completely classified).
In~\S\ref{section:Hopf} we study
automorphism groups of
Hopf surfaces. In~\S\ref{section:Kodaira} we study
automorphism groups of (secondary) Kodaira surfaces.
In \S\ref{section:non-negative} we study automorphism groups
of other minimal surfaces of non-negative Kodaira dimension, and prove
Theorems~\ref{theorem:Aut} and~\ref{theorem:main}.

Our general strategy is to consider the compact complex surfaces according to
Enriques--Kodaira classification.
Note that some
of our theorems follow from more general results
of I.\,Mundet i Riera,
cf. Theorems~\ref{theorem:class-VII} and~\ref{theorem:Riera}
(and also the discussion in the end of~\S\ref{section:class-VII}).
Similarly, some other results are implied by~\cite{Kim}.
We also point out that Jordan property always holds for the connected
component of the identity in the automorphism group of an arbitrary connected compact complex manifold
by~\mbox{\cite[Theorem~7]{Popov-LieGroups}}.

We are grateful to M.\,Brion, S.\,Nemirovski, and M.\,Verbitsky for useful discussions.
Special thanks go to the referee for a careful reading of our paper.

\section{Jordan property}
\label{section:Jordan}

In this section we collect some group-theoretic properties related to the Jordan property, and prove
a couple of auxiliary results about them.
We start by recalling a useful result that is very well known (see for instance \cite[\S4.4]{Springer1977}).

\begin{theorem}\label{theorem:PGL2}
Let $G\subset\Aut(\PP^1)\cong\PGL_2(\CC)$ and $\tilde{G}\subset\GL_2(\CC)$ be finite subgroups.
Then~$G$ is either cyclic, or dihedral,
or isomorphic to one of the groups~$\mathfrak{A}_4$, $\mathfrak{S}_4$, or~$\mathfrak{A}_5$.
In particular, the group $G$ has a cyclic subgroup of index at most~$12$,
and the group $\tilde{G}$ has an abelian subgroup of index at most~$12$. Furthermore, if $|G|$ is odd,
then $G$ is cyclic, and if $|\tilde{G}|$ is odd, then~$\tilde{G}$ is abelian.
\end{theorem}

Apart from the Jordan property, one can consider other restrictions formulated in
terms of finite subgroups of a given group.

\begin{definition}
We say that a group $\Gamma$
\emph{has bounded finite subgroups}
if there exists a constant $B=B(\Gamma)$ such that
for any finite subgroup
$G\subset\Gamma$ one has $|G|\leqslant B$.
\end{definition}

The following result is due to
H.\,Minkowski, see for instance~\cite[Theorem~1]{Serre2007}.

\begin{theorem}
\label{theorem:Minkowski}
For every $n$ the group $\GL_n(\QQ)$ has bounded finite subgroups.
\end{theorem}

\begin{definition}\label{definition:strongly-Jordan}
We say that a group~$\Gamma$
is \emph{strongly Jordan}
if it is Jordan, and there exists a constant $R=R(\Gamma)$ such that
every finite subgroup
in $\Gamma$ is generated by at most $R$ elements.
\end{definition}

Note that Definition~\ref{definition:strongly-Jordan} is equivalent to
a similar definition in~\cite{BandmanZarhin2015}.
An example of a strongly Jordan group is given by~$\GL_n(\CC)$.
This follows from
the fact that
every abelian subgroup of $\GL_n(\CC)$ is conjugate to a group that consists
of diagonal matrices.
Note however that even the group
$\CC^*$ contains \emph{infinite} abelian subgroups
of arbitrarily large rank.

The following elementary result will be useful to study Jordan property.

\begin{lemma}
\label{lemma:group-theory}
Let
$$
1\longrightarrow\Gamma'\longrightarrow\Gamma\longrightarrow\Gamma''
$$
be an exact sequence of groups. Then the following assertions
hold.
\begin{itemize}
\item[(i)] If $\Gamma'$ is Jordan (respectively, strongly Jordan) and $\Gamma''$ has bounded finite
subgroups, then~$\Gamma$ is Jordan (respectively, strongly Jordan).

\item[(ii)] If $\Gamma'$ has bounded finite subgroups and $\Gamma''$ is strongly Jordan,
then $\Gamma$ is strongly Jordan.
\end{itemize}
\end{lemma}
\begin{proof}
Assertion~(i) is obvious.
For assertion~(ii) see~\cite[Lemma~2.8]{Prokhorov-Shramov-2013}
or~\mbox{\cite[Lemma~2.2]{BandmanZarhin2015}}.
\end{proof}

It is easy to see that if~$\Gamma_1$ is a subgroup
of finite index in~$\Gamma_2$, then~$\Gamma_2$ is Jordan (respectively, strongly Jordan)
if and only so is~$\Gamma_1$.
At the same time
Jordan property, as well as strong Jordan property,
does not behave well with respect to quotients by infinite groups.
Namely, a quotient of a strongly Jordan group
by its subgroup may fail to be Jordan or to have all of its finite subgroups generated by a bounded number
of elements.
In spite of this we will be able to control the properties
of some quotients by infinite groups that will be important for us.

\begin{lemma}\label{lemma:diagonal-quotient-Z-1}
Let $A$ be an abelian group whose torsion subgroup $A_{\mathrm{t}}$
is isomorphic to~\mbox{$(\QQ/\ZZ)^n$}, and let
$\Lambda\subset A$ be a subgroup isomorphic to~$\ZZ^m$.
Then the quotient group~\mbox{$\Gamma=A/\Lambda$} is strongly Jordan.
\end{lemma}
\begin{proof}
The group $\Gamma$ is abelian and thus Jordan. Let $V\subset \Gamma$ be a finite
subgroup and let~\mbox{$\tilde V\subset A$} be its preimage. Clearly, $\tilde V$ is finitely generated
and can be decomposed into a direct product
$\tilde V=\tilde V_{\mathrm{t}}\times \tilde V_{\mathrm{f}}$ of its torsion and
torsion free parts.
In particular, $\tilde{V}_{\mathrm f}$ is a free abelian group.
Since $\tilde V_{\mathrm{f}}/(\tilde V_{\mathrm{f}}\cap \Lambda)$ is a finite
group,
one has
$$
\rk \tilde V_{\mathrm{f}}=\rk (\tilde V_{\mathrm{f}}\cap \Lambda)\le \rk
\Lambda=m.
$$
The group $\tilde V_{\mathrm{t}}$ is contained in $A_{\mathrm{t}}\cong
(\QQ/\ZZ)^n$
and so it can be generated by $n$ elements. Thus~$\tilde V$
can be generated by $n+m$ elements, and the images of these elements in $\Gamma$
generate the subgroup~$V$.
\end{proof}

\begin{lemma}\label{lemma:central-subgroup}
Let
\begin{equation}\label{eq:central-subgroup}
1\longrightarrow\Gamma'\longrightarrow\Gamma\longrightarrow\Gamma''
\end{equation}
be an exact sequence of groups. Suppose that $\Gamma'$ is central in $\Gamma$
(so that in particular~$\Gamma'$ is abelian) and
there exists a constant $R$ such that every finite subgroup of $\Gamma'$
is generated by at most $R$ elements.
Suppose also that there exists a constant $J$ such that for
every finite subgroup $G\subset\Gamma''$ there is a cyclic subgroup
$C\subset G$ of index at most $J$ (so that in particular
$\Gamma''$ is strongly Jordan). Then
the group $\Gamma$ is strongly Jordan.
\end{lemma}
\begin{proof}
Let $G\subset\Gamma$ be a finite subgroup. The exact
sequence~\eqref{eq:central-subgroup}
induces an exact sequence of groups
$$
1\longrightarrow G'\longrightarrow G\longrightarrow G'',
$$
where $G'$ is a subgroup of $\Gamma'$ (in particular, $G'$ is abelian),
while $G''$ is a subgroup of~$\Gamma''$. There is a subgroup $\bar{G}\subset G$
of index at most $J$ such that $\bar{G}$ contains $G'$, and the
quotient~$\bar{G}/G'$ is a cyclic group. To prove that the group $\Gamma$ is Jordan it is
enough
to check that $\bar{G}$ is an abelian group. The latter follows from the fact
that
$G'$ is a central subgroup of~$\bar{G}$.

The assertion about the bounded number of generators is obvious.
\end{proof}

\begin{lemma}\label{lemma:GL-quotient-Z}
Let $\Lambda$ be a finitely generated central subgroup of $\GL_2(\CC)$.
Then the quotient group $\Gamma=\GL_2(\CC)/\Lambda$ is strongly Jordan.
\end{lemma}
\begin{proof}
We have an exact sequence of groups
$$
1\longrightarrow
\CC^*/\Lambda\longrightarrow\Gamma\longrightarrow\PGL_2(\CC)\longrightarrow 1.
$$
The group $\CC^*/\Lambda$ is a central subgroup of $\Gamma$.
Also, the group $\CC^*/\Lambda$ is strongly Jordan
by Lemma~\ref{lemma:diagonal-quotient-Z-1}.

On the other hand, we know from the classification of finite subgroups
of $\PGL_2(\CC)$ that every finite subgroup therein contains a cyclic subgroup
of bounded index, see Theorem~\ref{theorem:PGL2}.
Therefore, the assertion follows from
Lemma~\ref{lemma:central-subgroup}.
\end{proof}

We will need the following simple observation in~\S\ref{section:Hopf}.

\begin{lemma}\label{lemma:characteristic-subgroup}
Let $\Gamma$ be a group containing a subgroup $\Lambda\cong\ZZ$ of finite index.
Then there is a subgroup $\Lambda_0\cong\ZZ$ in $\Lambda$ that is characteristic in $\Gamma$.
\end{lemma}

\begin{proof}
The intersection
$$
\Lambda_0=\bigcap\limits_{\theta\in\Aut(\Gamma)}\theta(\Lambda)
$$
is a characteristic subgroup in $\Gamma$. Therefore, it is enough to check that
$\Lambda_0$ is not a trivial group.

Denote $r=[\Gamma:\Lambda]$. For any $\theta\in\Aut(\Gamma)$ the group $\theta(\Lambda)$ has index
$r$ in $\Gamma$. Hence the index of the intersection $\Lambda\cap\theta(\Lambda)$ in $\Lambda$
is at most $r$. This means that the intersection of $\Lambda$ with \emph{all}
groups $\theta(\Lambda)$, $\theta\in\Aut(\Gamma)$, contains an intersection of all these subgroups in~$\Lambda$.
Since a subgroup of given index in $\Lambda$ is unique, we see that the latter intersection is non-trivial.
\end{proof}

Most of the groups we will be working with in
the remaining part of the paper will be strongly Jordan.
However, we will only need to check Jordan property for them due
to the following result.

\begin{theorem}[{\cite[Theorem~1.3]{Riera-OddCohomology}}]
\label{theorem:Riera-generators}
For any compact manifold $X$ there is a constant~$R$ such that
every finite group acting effectively by diffeomorphisms of~$X$ can be generated by at most~$R$
elements.
\end{theorem}

\section{Minimal surfaces}
\label{section:minimal-surfaces}

In this section we recall the basic properties
of compact complex surfaces. Everything here (as well as in
\S\ref{section:automorphisms} below) is well known to experts,
but in some important cases we provide proofs for the reader's convenience.

A complex surface is a complex manifold of (complex) dimension~$2$.
Starting from this point we will always assume that our complex surfaces
are connected.
Throughout the paper $\KKK_X$ denotes
the canonical line bundle of a compact complex surface~$X$.
One has~\mbox{$\cc_1(\KKK_X)=-\cc_1(X)$}.
By $\ad(X)$ we denote the algebraic dimension of~$X$, i.e. the transcendence
degree of the field of meromorphic functions on~$X$.

\begin{definition}
Let $X$ and $Y$ be compact complex surfaces.
A proper holomorphic map~\mbox{$f\colon X\to Y$} is said to be a \textit{proper modification}
if there are closed
analytic subsets~\mbox{$Z_1\subsetneq X$} and $Z_2\subsetneq Y$ such that
the restriction $f_{X\setminus Z_1}\colon X\setminus Z_1\to Y\setminus Z_2$
is biholomorphic.
A \textit{birational} (or \textit{bimeromorphic}) \textit{map} $X\dashrightarrow
Y$ is an equivalence class
of diagrams
\[
\xymatrix{
&Z\ar[dr]^g\ar[dl]_f&
\\
X\ar@{-->}[rr]&&Y
}
\]
where $f$ and $g$ are proper modifications, modulo natural equivalence relation.
\end{definition}

Birational maps from a given compact complex surface
$X$ to itself form a group, which we will denote by~$\Bir(X)$.
As usual, we say that two complex surfaces are birationally
equivalent, or just birational,
if there exists a birational map between them.

\begin{remark}
If $X$ and $Y$ are birationally equivalent compact complex surfaces, then
the fields of meromorphic functions on $X$ and $Y$
are isomorphic.
The converse is not true if the algebraic dimension of
$X$ (and thus also of $Y$) is less than~$2$.
\end{remark}

There are easy ways to find whether a given compact complex surface is projective.

\begin{theorem}[{see~\cite[Corollary~IV.6.5]{BHPV-2004}}]
A compact complex surface $X$ is projective if and only
if~\mbox{$\ad(X)=2$}. In particular, any compact complex surface birational to a projective
one is itself projective.
\end{theorem}

\begin{lemma}[{see \cite[Theorem~IV.6.2]{BHPV-2004}}]
\label{lemma:projective}
Let $X$ be a compact complex surface. Suppose that there is a line bundle
$\mathcal{L}$ on $X$ such that $\mathcal{L}^2>0$. Then $X$ is projective.
\end{lemma}

A \emph{$(-1)$-curve} on a compact complex surface is a smooth
rational curve with self-intersection equal to~$-1$.
A compact complex surface is \emph{minimal} if it does not contain
$(-1)$-curves. The following fact is well known, see e.g. Corollary~III.2.4,  Claim on p.~99,
and the first paragraph
of~\S\,VI.7 in~\cite{BHPV-2004}). For convenience of the reader we provide its short proof.

\begin{proposition}\label{proposition:Bir-vs-Aut}
Let $X$ be a minimal surface. Suppose that $X$ is neither
rational nor ruled. Then every birational map from an arbitrary compact complex
surface $X'$ to $X$ is a proper modification. In particular, $X$ is the unique minimal model in its
class of birational equivalence, and
$\Bir(X)=\Aut(X)$.
\end{proposition}
\begin{proof}
Suppose that
\[
\xymatrix@R=8pt{
&Z\ar[dr]^f\ar[dl]_g&
\\
X'\ar@{-->}[rr]&&X
}
\]
is a birational map that is not a proper modification.
We may assume that there are no $(-1)$-curves that are simultaneously contracted by $f$ and $g$.
Then
there exists a $(-1)$-curve~$C$ in~$Z$ contracted by $g$ but not contracted by~$f$.
Thus $C$ meets a one-dimensional fiber~$f^{-1}(x)$ for some point~\mbox{$x\in X$}, since otherwise $X$ would contain a
$(-1)$-curve.

First, we consider the case when the surface $X$ is projective.
Since $X$ is minimal and not ruled,  the canonical class $K_X$ must be numerically
effective  \cite[Theorem~VI.2.1]{BHPV-2004}.
Write
\[
K_Z\sim f^*K_X+\sum a_i E_i,
\]
where $E_i$ are $f$-exceptional curves
and $a_i$ are positive integers. Since $K_Z\cdot C<0$ and~\mbox{$f^*K_X\cdot C\ge 0$}, we have $\sum a_i E_i\cdot C<0$.
Thus $C$ is a component of the $f$-exceptional locus. This  contradicts our assumptions.

Now we consider the case when the surface $X$ is  not projective.
Contracting $(-1)$-curves in $f^{-1}(x)$ consecutively, we get a surface
$S$ with a proper modification $h\colon Z\to S$, and
a proper modification~\mbox{$t\colon S\to X$} such that $C_1=h(C)$ is a
$(-1)$-curve
and there exists
another $(-1)$-curve $C_2$ meeting $C_1$ and contracted by $t$.
If $C_1\cdot C_2>1$, then $(C_1+C_2)^2>0$ and the surface $S$ is
projective by Lemma~\ref{lemma:projective}.
Assume that $C_1\cdot C_2=1$. Then for $n\gg0$ we have
\[
\cc_1\left(\KKK_S\otimes\OOO_S(-nC_1-nC_2)\right)^2=\cc_1(S)^2+4n>0,
\]
so that the surface $S$ is again projective by Lemma~\ref{lemma:projective}.
The obtained contradiction completes the proof.
\end{proof}

Given a compact complex surface $X$, we can consider its
\emph{pluricanonical linear systems}~$|\KKK_X^{\otimes m}|$.
If such a linear system is not empty for $m\gg 0$, it defines a rational \emph{pluricanonical
map}. The dimension of its image is called the \emph{Kodaira dimension} of $X$ and
is denoted by~$\varkappa(X)$. If the linear system $|\KKK_X^{\otimes m}|$ is empty for all~$m>0$,
we put~\mbox{$\varkappa(X)=-\infty$}.
By $\bb_i(X)$ we denote the $i$-th
Betti number of~$X$. By~$\hh^{p,q}(X)$ we denote the Hodge numbers~\mbox{$\hh^{p,q}=\dim H^q(X,\Omega^p_X)$},
where~$\Omega^p_X$ is the sheaf of holomorphic $p$-forms on~$X$.

The following is the famous Enriques--Kodaira classification of
compact complex surfaces, see e.g.~\cite[Chapter~VI]{BHPV-2004}.

\begin{theorem}\label{theorem:classification}
Let $X$ be a minimal compact complex surface.
Then $X$ is of one of the following types.
\par\medskip\noindent
{\rm
\setlength{\extrarowheight}{1pt}
\newcommand{\heading}[1]{\multicolumn{1}{c|}{#1}}
\newcommand{\headingl}[1]{\multicolumn{1}{c}{#1}}
\begin{tabularx}{\textwidth}{c|l|X|X|X}
$\varkappa(X)$&\heading{type} & \heading{$\ad(X)$} & \heading{$\bb_1(X)$} &
\headingl{$\chit(X)$}
\\\hline
\multirow3{*}{$-\infty$}&rational surfaces& $2$& $0$ &$3$, $4$
\\
& ruled surfaces of genus $g>0$& $2$ &$2g$&$4(1-g)$
\\
& surfaces of class VII& $0$, $1$&$1$ &$\ge 0$
\\\hline
\multirow6{*}{$0$}& complex tori& $0$, $1$, $2$&$4$& $0$
\\
& K3 surfaces& $0$, $1$, $2$& $0$&$24$
\\
& Enriques surfaces& $2$& $0$&$12$
\\
&bielliptic surfaces & $2$& $2$& $0$
\\
&primary Kodaira surfaces& $1$&$3$& $0$
\\
&secondary Kodaira surfaces& $1$& $1$& $0$
\\\hline
\multirow1{*}{$1$}&properly elliptic surfaces& $1$, $2$&& $\ge 0$
\\\hline
\multirow1{*}{$2$}&surfaces of general type & $2$&$\equiv 0\mod 2$& $> 0$
\\
\end{tabularx}
}
\end{theorem}

\section{Automorphisms}
\label{section:automorphisms}

In this section
we recall some important general facts about automorphisms
of complex spaces.

Let $U$ be a reduced complex space, see e.g. \cite{Serre1956} or
\cite{Malgrange1968}
for a definition and basic properties.
Recall that a complex space is called \emph{irreducible}
if it cannot be represented as a union of two proper closed analytic subsets.
We denote by $T_{P, U}$ the Zariski
tangent space
(see \cite[\S2]{Malgrange1968})
to $U$ at a point $P\in U$.
If a group $\Gamma\subset\Aut(X)$ has a fixed point~\mbox{$P\in X$}, then
$\Gamma$ naturally acts on the local ring $\OOO_{P,X}$ and the tangent space $T_{P,X}$
so that the action on $T_{P,X}$ is linear.

The following fact is well-known (see e.g. \cite{Cartan} or \cite[\S2.2]{Akhiezer1995}).
\begin{theorem}\label{theorem:linearization}
Let $X$ be a Hausdorff (reduced)
complex space,
and~\mbox{$\Gamma\subset\Aut(X)$} be a finite group.
Suppose
that $\Gamma$ has a fixed point $P$ on $X$.
Then there exist
$\Gamma$-invariant neighborhoods $U$ of $P$
in $X$ and $V$ of $0$ in $T_{P,X}$, and a $\Gamma$-equivariant closed embedding~\mbox{$U \hookrightarrow V$}.
\end{theorem}

\begin{corollary}
\label{corollary:fixed-point}
Let $X$ be an irreducible Hausdorff reduced complex space,
and $\Gamma\subset\Aut(X)$ be a finite group. Suppose
that $\Gamma$ has a fixed point $P$ on $X$. Then the natural
representation
$$
\Gamma\longrightarrow\GL\big(T_{P,X}\big)
$$
is faithful.
\end{corollary}

\begin{proof}
Choose an arbitrary transformation $f$ from the kernel of the action of $\Gamma$ on $T_{P,X}$.
By Theorem~\ref{theorem:linearization}, there exists a neighborhood
$U$ of $P$ in $X$ such that $f$ restricts to the identity transformation on $U$.

The fixed point locus $\Fix(f)$ of $f$ is a closed analytic subset in $X$.
Indeed, since~$X$ is Hausdorff, it can be covered by $f$-invariant
charts isomorphic to open subsets of~$\CC^N$ for some positive integer~$N$,
and in every such chart the fixed point locus is given by vanishing of certain
equations. Since $\Fix(f)$ contains the open subset $U$ and $X$ is irreducible, we conclude that
$\Fix(f)=X$. This means that the kernel of the action of $\Gamma$ on $T_{P,X}$
is trivial.
\end{proof}

\begin{remark}
One cannot drop the assumption that $X$ is irreducible in
Corollary~\ref{corollary:fixed-point}. Indeed, the assertion fails for
the variety given by equation~\mbox{$xy=0$} in~$\mathbb{A}^2$
with coordinates~$x$ and~$y$, the point $P$ with
coordinates~\mbox{$x=1$} and~\mbox{$y=0$}, and the group $\Gamma\cong\ZZ/2\ZZ$
whose generator acts by~\mbox{$(x,y)\mapsto (x,-y)$}.
Similarly, the assertion fails for the simplest example
of a non-Hausdorff reduced complex space, namely, for two copies of $\mathbb{A}^1$
glued along the common open subset~\mbox{$\mathbb{A}^1\setminus\{0\}$}, and the natural
involution acting on this space.
\end{remark}

Corollary~\ref{corollary:fixed-point} easily implies the following result.

\begin{corollary}
\label{corollary:fixed-point1}
Let $X$ be an irreducible Hausdorff reduced complex space,
and~\mbox{$\Delta\subset\Aut(X)$} be a subgroup. Suppose
that $\Delta$ has a fixed point $P$ on $X$, and let
$$
\varsigma\colon\Delta\longrightarrow\GL\big(T_{P,X}\big)
$$
be the natural representation.
Suppose that there is a subgroup $\Gamma\subset\Delta$ of finite
index such that the restriction $\varsigma\vert_{\Gamma}$ is a group monomorphism.
Then $\varsigma$ is an embedding as well.
\end{corollary}
\begin{proof}
Let $\Delta_0\subset\Delta$ be the kernel of $\varsigma$. Since
$[\Delta:\Gamma]<\infty$, we see that $\Delta_0$ is finite.
Thus~$\Delta_0$ is trivial by Corollary~\ref{corollary:fixed-point}.
\end{proof}

Another application of Corollary~\ref{corollary:fixed-point} is as follows.

\begin{lemma}\label{lemma:rational-curve}
Let $X$ be a compact complex surface.
Suppose that there is a finite non-empty $\Aut(X)$-invariant set $\mathcal{S}$ of curves on $X$
such that $\mathcal{S}$ does not contain smooth elliptic curves.
Then the group $\Aut(X)$ is Jordan.
\end{lemma}
\begin{proof}
Let $C$ be one of the curves from $\mathcal{S}$. Then the group $\Aut_C(X)$
of automorphisms of~$X$ that preserve the curve $C$ has finite
index in $\Aut(X)$. Since $C$ is not a smooth elliptic curve,
there is
a constant $B=B(C)$ such that every finite subgroup of~\mbox{$\Aut_C(X)$}
contains a subgroup of index at most $B$ that fixes some point on $C$.
Indeed, if $C$ is singular, this is obvious; if $C$ is a smooth rational curve, this follows from
Theorem~\ref{theorem:PGL2}. If~$C$
is a smooth curve of genus $g\ge 2$, this follows from the fact that the index
of the kernel of the action on $C$ in the group $\Aut_C(X)$ is at most $|\Aut(C)|$,
which does not exceed the Hurwitz bound~$84(g-1)$, see for instance~\mbox{\cite[Exercise~IV.2.5]{Hartshorne}}.
Now Corollary~\ref{corollary:fixed-point} implies that every finite subgroup of
$\Aut_C(X)$ contains a subgroup of index at most $B$ that is embedded into $\GL_2(\CC)$.
Therefore, the group $\Aut_C(X)$ is Jordan by Theorem~\ref{theorem:Jordan}, and hence
the group $\Aut(X)$ is Jordan as well.
\end{proof}

Using Theorem~\ref{theorem:linearization}, one can also deduce the following facts.

\begin{corollary}\label{corollary:fixed-locus:smooth}
Let $X$ be a complex manifold,
and $\Gamma\subset\Aut(X)$ be a finite group.
Then the fixed point locus of $\Gamma$ is a closed submanifold.
\end{corollary}

\begin{proof}
Let $Y$ be the fixed point locus of $\Gamma$. It is obvious that $Y$ is a closed subset of $X$.

Choose a point $P\in Y$. We know from
Theorem~\ref{theorem:linearization} that there is a $\Gamma$-equivariant closed embedding
$U\hookrightarrow V$ for some $\Gamma$-invariant neighborhoods $U$ of $P$ in $X$ and $V$ of $0$
in~$T_{P, X}$. Under this embedding the neighborhood $U_Y=Y\cap U$ of $P$ in $Y$ is isomorphically mapped
onto the fixed point locus $F$ of $\Gamma$ in $V$. Since the action of
$\Gamma$ on $T_{P, X}$ is linear, we conclude that $F$ is
an intersection of $V$ with some linear subspace of $T_{P, X}$. In particular, we see that $F$ is smooth at~$0$,
which implies that $Y$ is smooth at $P$.
\end{proof}

Note that the fixed point locus discussed in Corollary~\ref{corollary:fixed-locus:smooth} may consist
of several connected components of different dimensions.

\begin{corollary}\label{corollary:fixed-locus}
Let $X$ be a complex manifold,
and $\Gamma\subset\Aut(X)$ be a finite group. Suppose
that $\Gamma$ has a fixed point $P$ on $X$ and let
$T\subset T_{P,X}$ be the maximal subspace on which the action of $\Gamma$ is trivial.
Then there exists a $\Gamma$-invariant submanifold $Y\subset X$ containing~$P$ such
that~\mbox{$T=T_{P,Y}$} and the action of~$\Gamma$ on $Y$ is trivial.
\end{corollary}

\begin{proof}
The fixed point locus $Y$ of $\Gamma$ is a closed submanifold
by Corollary~\ref{corollary:fixed-locus:smooth}.
Clearly, one has $T_{P,Y}\subset T$. On the other hand, by Theorem~\ref{theorem:linearization}
we have~\mbox{$\dim T_{P,Y}=\dim T$}.
\end{proof}

\section{Non-projective surfaces with $\chit(X)\neq 0$}
\label{section:class-VII}

In this section we will (mostly) work with non-projective compact complex surfaces
$X$ with $\chit(X)\neq 0$.
In this case, by the Enriques--Kodaira classification (see Theorem \ref{theorem:classification})
one has $\chit(X)> 0$.
The main purpose of this section is to prove the following result.

\begin{theorem}\label{theorem:class-VII}
Let $X$ be a non-projective compact complex surface with $\chit(X)\neq 0$.
Then the group $\Aut(X)$ is Jordan.
\end{theorem}

Recall that an \emph{algebraic reduction} of a compact complex surface
$X$ with $\ad(X)=1$ is the morphism $\pi\colon X\to B$ to a curve $B$ obtained as follows.
We start with a meromorphic map~\mbox{$X\dasharrow\PP^1$} defined by a non-constant
meromorphic function, regularize it by blow ups, and apply the Stein
factorization to the regularization. One can check that
the obtained morphism provides a holomorphic elliptic fibration $\pi$ on $X$. We refer
the reader to~\mbox{\cite[Proposition~VI.5.1]{BHPV-2004}} for details.

\begin{lemma}\label{lemma-elliptic-curves}
Let $X$ be a non-projective compact complex surface.
If $X$ contains an irreducible curve $C$ which is not a smooth elliptic curve,
then
the group $\Aut(X)$ is Jordan.
\end{lemma}

\begin{proof}
We claim that
the surface $X$ contains at most a finite number of such curves.
Indeed, if $\ad(X)=0$, then $X$ contains at most a
finite number of curves at all, see~\mbox{\cite[Theorem~IV.8.2]{BHPV-2004}}.
If $\ad(X)=1$, then all curves on $X$ are contained
in the fibers of the algebraic reduction by Lemma~\ref{lemma:projective}.
The latter fibration is elliptic,
so every non-elliptic curve is contained in one of its degenerate
fibers. Now the assertion follows from Lemma~\ref{lemma:rational-curve}.
\end{proof}

\begin{lemma}\label{lemma:a-1}
Let $X$ be a compact complex surface
with $\chit(X)\neq 0$.
If $\ad(X)=1$, then the group $\Aut(X)$ is Jordan.
\end{lemma}
\begin{proof}
Let $\pi\colon X\to B$ the algebraic reduction, so that $B$ is a
smooth curve and $\pi$ is an elliptic fibration.
Since
$\chit(X)\neq 0$,
the fibration $\pi$ has
at least one fiber $X_b$ such that $F=(X_b)_{\red}$
is not a smooth elliptic curve. So
the group $\Aut(X)$ is Jordan by Lemma~\ref{lemma-elliptic-curves}.
\end{proof}

For every compact complex surface $X$, we denote
by $\overline{\Aut}(X)$ the subgroup of $\Aut(X)$ that consists of all elements acting trivially
on $H^*(X,\QQ)$. This is a normal subgroup of $\Aut(X)$, and the quotient group
$\Aut(X)/\overline{\Aut}(X)$ has bounded finite subgroups
by Theorem~\ref{theorem:Minkowski}. Thus Lemma~\ref{lemma:group-theory}(i)
implies
that the group $\Aut(X)$ is Jordan if and only if~\mbox{$\overline{\Aut}(X)$} is
Jordan.

\begin{lemma}\label{lemma-b2-fixed-points}
Let $X$ be a compact complex surface. Suppose that every irreducible curve
contained in $X$
is a smooth elliptic curve. Let $g\in \overline{\Aut}(X)$ be a non-trivial element of
finite
order, and $\Xi_0(g)$ be the set of all isolated fixed points of $g$. Then
$$
|\Xi_0(g)|=\chit(X).
$$
\end{lemma}
\begin{proof}
The fixed locus $\Xi(g)$ of~$g$ is a disjoint union $\Xi_0(g)\sqcup\Xi_1(g)$,
where
$\Xi_1(g)$ is of pure dimension $1$.
Note that
the curve $\Xi_1(g)$ is smooth by Corollary~\ref{corollary:fixed-locus:smooth}, so that
every irreducible component of $\Xi_1(g)$ is its connected
component.

We see that every connected component of $\Xi_1(g)$ is a smooth elliptic
curve,
so that~\mbox{$\chit(\Xi_1(g))=0$}. On the other hand, one has
$$
\chit(\Xi(g))=\chit(X)
$$
by the topological Lefschetz fixed point formula,
see \cite[Proposition~5.3.11]{Dieck1979}.
Therefore, we have
$$
\chit(X)=\chit(\Xi(g))=\chit(\Xi_0(g))+\chit(\Xi_1(g))=\chit(\Xi_0(g))=|\Xi_0(g)|. \qedhere
$$
\end{proof}

\begin{lemma}\label{lemma-cyclic-subgroups}
Let $X$ be a compact complex surface
with $\chit(X)\neq 0$. Suppose that every irreducible curve contained in
$X$
is a smooth elliptic curve.
Let $G\subset \overline{\Aut}(X)$ be a finite subgroup.
If $G$ contains a non-trivial normal cyclic subgroup,
then $G$ contains an abelian subgroup of index at most $12\chit(X)$.
\end{lemma}
\begin{proof}
Let $N\subset G$ be a non-trivial normal cyclic subgroup.
By Lemma \ref{lemma-b2-fixed-points}
the group $N$ has exactly
$\chit(X)>0$
isolated fixed points on~$X$
(and maybe also several curves that consist of fixed points).
Since $N$ is normal in $G$, the group $G$ permutes these points.
Thus there exists a subgroup of
index at most $\chit(X)$ in $G$ acting on $X$ with a fixed point.
Now the assertion follows from Corollary~\ref{corollary:fixed-point}
and Theorem~\ref{theorem:PGL2} (cf.~\mbox{\cite[Corollary~2.2.2]{Prokhorov-Shramov-JCr3}}).
\end{proof}

\begin{lemma}\label{lemma:at-least-one-curve}
Let $X$ be a compact complex surface
with $\ad(X)=0$ and $\chit(X)\neq 0$.
If $X$ contains at least one curve, then $\Aut(X)$ is Jordan.
\end{lemma}
\begin{proof}
It is enough to prove that the group
$\overline{\Aut}(X)$ is Jordan. The surface $X$ contains at most a finite number
of curves by~\cite[Theorem~IV.8.2]{BHPV-2004}.
By Lemma \ref{lemma-elliptic-curves} we may assume that all these curves
are smooth and elliptic.
Let $C_1,\dots, C_m$ be all curves on $X$, and let
$\Aut^\sharp(X)\subset \overline{\Aut}(X)$ be the stabilizer of~$C_1$.
Clearly, the subgroup $\Aut^\sharp(X)$ has index at most $m$ in
$\overline{\Aut}(X)$,
so it is sufficient to prove that $\Aut^\sharp(X)$ is Jordan.
For any finite subgroup $G\subset \Aut^\sharp(X)$ we have an exact sequence
\[
1\longrightarrow\Gamma \longrightarrow G\longrightarrow\Aut(C_1),
\]
where $\Gamma$ is the kernel of the action of $G$ on $C_1$.

Let $P$ be a point on $C_1$.
Then $\Gamma\subset\GL(T_{P,X})$ by Corollary~\ref{corollary:fixed-point},
and $\Gamma$ has a trivial one-dimensional subrepresentation
in $T_{P,X}$ corresponding to the tangent space~$T_{P,C_1}$. This implies that
$\Gamma$
is a cyclic group. If $\Gamma=\{1\}$, then
$G$ is contained in $\Aut(C_1)$. Since $C_1$ is an elliptic curve,
the group $G$ has an abelian subgroup of index at most $6$.
If $\Gamma\neq \{1\}$, then
$G$ has an abelian subgroup of index at most $12\chit(X)$
by Lemma \ref{lemma-cyclic-subgroups}.
Therefore, in both cases $G$ also has a normal abelian subgroup of bounded
index.
\end{proof}

In the following lemmas we will deal with compact complex surfaces $X$
that contain no curves. In particular, this implies that $\ad(X)=0$. Furthermore,
we conclude that the action of any finite subgroup $G$ of $\Aut(X)$
is free in codimension one, that is, there exists a finite
subset $\Xi\subset X$ such that the action of $G$ on $X\setminus\Xi$ is free.

\begin{lemma}\label{lemma:no-even-order}
Let $X$ be a compact complex surface
with $\chit(X)\neq 0$. Suppose that $X$ contains no curves.
Then the group
$\overline{\Aut}(X)$ has no elements of even order.
\end{lemma}
\begin{proof}
Let $g\in \overline{\Aut}(X)$ be an element of order $2$
(such elements always exist provided that there are elements
of even order).

First assume that $\varkappa(X)=-\infty$.
We have $\bb_1(X)=1$ and $\bb_2(X)=\chit(X)>0$ (see Theorem~\ref{theorem:classification}).
Moreover, we know that $\hh^{2,0}(X)=0$ because $\varkappa(X)=-\infty$.
Hodge relations (see e.g. \cite[\S\,IV.2]{BHPV-2004})
give us
\[
\hh^{0,1}(X)=1,\quad \hh^{1,0}(X)=0, \quad \text{and}\quad \hh^{2,0}(X)=\hh^{0,2}(X)=0.
\]
Thus, one has $\chi(\OOO_X)=0$, and the canonical embedding $H^1(X,\OOO_X)\hookrightarrow H^1(X,\CC)$ is
an isomorphism.
In particular, the element $g$ acts trivially on $H^1(X,\OOO_X)$.
We also know that there are no curves consisting of $g$-fixed points. Therefore,
the holomorphic Lefschetz fixed point formula (see e.~g.~\mbox{\cite[\S3.4]{Griffiths-Harris-1978}})
can be written as follows:
\begin{equation*}
\sum\limits_{P\in\Fix(g)}\frac{1}{\det\left(\operatorname{Id}-g_P\right)}= \sum\limits_{q=0}^2 (-1)^q
\operatorname{tr} g^* |_{H^q(X,\OOO_X)}=\bb_0(X)-\hh^{0,1}(X)+\hh^{0,2}(X)=0,
\end{equation*}
where $\Fix(g)$ is the fixed point locus of $g$, and $g_P\colon T_{P,X}\to T_{P,X}$ is the differential
of $g$ at a fixed point $P$. Since the order of $g$ equals $2$, one has $g_p=-\operatorname{Id}$,
because otherwise there exists an analytic germ of a curve in a neighborhood of $P$ in $X$ that consists
of fixed points of~$g$ by Corollary~\ref{corollary:fixed-locus}.
Hence
$$
\frac{|\Fix(g)|}{4}=\sum\limits_{P\in\Fix(g)}\frac{1}{\det\left(\mathrm{Id}-g_P\right)}=0.
$$
Thus, we conclude that $g$ has no fixed points at all. The latter contradicts Lemma~\ref{lemma-b2-fixed-points}.

Now assume that $\varkappa(X)\ge 0$.
Since $\ad(X)=0$, this implies that $\varkappa(X)=0$ and $X$ is a K3 surface
(see Theorem~\ref{theorem:classification}).
Therefore, one has $\chit(X)=24$ and $\chi(\OOO_X)=2$.
As above the holomorphic Lefschetz fixed point formula
shows that $g$ has exactly $8$ fixed points.
This again contradicts Lemma~\ref{lemma-b2-fixed-points}.
\end{proof}

\begin{lemma}\label{lemma-G-fixed-point}
Let $X$ be a compact complex surface
with $\chit(X)\neq 0$. Suppose that $X$ contains no curves.
Let $G\subset\overline{\Aut}(X)$ be a finite subgroup. Suppose that
$G$ has a fixed point on $X$. Then $G$ is cyclic.
\end{lemma}

\begin{proof}
Let $P\in X$ be a fixed point of $G$.
By Corollary~\ref{corollary:fixed-point} we have an embedding
$$
G\subset \GL(T_{P, X})\cong\GL_2(\CC).
$$
Since the group $G$ does not contain elements of order~$2$ by Lemma~\ref{lemma:no-even-order},
the order of $G$ is odd. Hence $G$ is
abelian by Theorem~\ref{theorem:PGL2}.
Recall that the action of $G$ is free in codimension one.
By Corollary~\ref{corollary:fixed-point}, the action of $G$ on~\mbox{$T_{P,X}\cong\CC^2$} is faithful.

Suppose that $G$ is not a cyclic group.
Since $G$ is abelian, its action on $\CC^2$ is diagonalizable and so
there exists a non-trivial element~\mbox{$g\in G$} such that
$g$ has an eigen-vector with eigen-value $1$ in $T_{P,X}$. By Corollary~\ref{corollary:fixed-locus}
there exists an analytic germ of a curve in a neighborhood of $P$ in $X$ that consists
of fixed points of~$g$. The latter is impossible since the action of $g$ is free in codimension one.
The obtained contradiction shows that the group $G$ is cyclic.
\end{proof}

\begin{lemma}\label{lemma:same-fixed-points}
Let $X$ be a compact complex surface
with $\chit(X)\neq 0$. Suppose that $X$ contains no curves.
Let $G\subset\overline{\Aut}(X)$ be a finite cyclic subgroup, and $g\in G$ be
a non-trivial element. Then $g$ has the same set of fixed points as $G$.
\end{lemma}
\begin{proof}
For an arbitrary element $h\in G$ denote by $\Fix(h)$ the fixed locus of
$h$. By Lemma~\ref{lemma-b2-fixed-points} one has
$$
|\Fix(h)|=\chit(X)
$$
for every non-trivial element $h$.

Let $f$ be a generator of $G$. Then
for some positive integer $n$ one has
$f^n=g$, so that
$$
\Fix(f)\subset\Fix(g).
$$
Therefore,
one has $\Fix(f)=\Fix(g)$, which means that every non-trivial
element of $G$ has one and the same set of fixed points.
\end{proof}

\begin{lemma}\label{lemma-G-decomposition}
Let $X$ be a compact complex surface
with $\chit(X)\neq 0$. Suppose that $X$ contains no curves.
Then every finite subgroup $G\subset\overline{\Aut}(X)$ is a union
$G=\bigcup_{i=1}^m G_i$
of cyclic subgroups such that $G_i\cap G_j=\{1\}$ for $i\neq j$.
\end{lemma}
\begin{proof}
Choose some representation of $G$ as a union $G=\bigcup_{i=1}^m G_i$,
where $G_i$ are cyclic groups that possibly have non-trivial intersections.
Let $G_1$ and $G_2$ be subgroups such that $G_1\cap G_2\neq \{1\}$.
Let $g\in G_1\cap G_2$ be a non-trivial element. By
Lemma~\ref{lemma-b2-fixed-points}
it has a fixed point, say $x$.
By Lemma \ref{lemma-G-fixed-point} the stabilizer $G_x$ is a cyclic group.
By Lemma~\ref{lemma:same-fixed-points} the groups $G_1$ and $G_2$ fix the point
$x$,
so that $G_1, G_2\subset G_x$.
Replacing $G_1$ and $G_2$ by $G_x$, we proceed to construct the required
system of subgroups by induction.
\end{proof}

\begin{lemma}\label{lemma:no-curves}
Let $X$ be a compact complex surface
with $\chit(X)\neq 0$.
Suppose that $X$ contains no curves.
Then there exists a constant $J=J(X)$ such that any finite subgroup $G\subset\Aut(X)$
contains a normal cyclic subgroup of index at most $J$. In particular, the group $\Aut(X)$ is Jordan.
\end{lemma}
\begin{proof}
It is enough to prove that any finite subgroup of $\overline{\Aut}(X)$ contains a normal cyclic subgroup of index at most $J$.
Let $G\subset \overline{\Aut}(X)$ be a finite subgroup.
Let $\Xi\subset X$ be the set of points with non-trivial stabilizers in $G$.

By Lemma~\ref{lemma-G-decomposition} the group $G$ is a union $G=\bigcup_{i=1}^m
G_i$
of cyclic subgroups such that~\mbox{$G_i\cap G_j=\{1\}$} for $i\neq j$.
We claim that the stabilizer
of any point $x\in \Xi$ is one of the groups $G_1,\dots, G_m$.
Indeed, choose a point $x\in\Xi$, and let $G_x$ be its stabilizer.
Then $G_x$ is a cyclic group by Lemma~\ref{lemma-G-fixed-point}.
Let $g_x$ be a generator of
$G_x$, and let $1\le r\le m$ be the index such that the group $G_r$
contains~$g_x$. Then $G_x\subset G_r$. Now Lemma~\ref{lemma:same-fixed-points}
implies
that $G_x=G_r$.

By Lemma~\ref{lemma-b2-fixed-points} every element of $G$ has exactly $\chit(X)$
fixed points.
The set $\Xi$ is a disjoint union of orbits of the group~$G$.
Therefore, for some positive integers $k_i$ one has
\[
|\Xi|= m \chit(X)= \sum_{i=1}^m k_i [G: G_i].
\]
Hence, for some $i$ we have $[G: G_i]\le \chit(X)$, i.e.
$G$ contains a cyclic subgroup $G_i$ of index at most~$\chit(X)$.
This implies that $G$ contains a normal cyclic subgroup of bounded index.
\end{proof}

Now we are ready to prove Theorem~\ref{theorem:class-VII}.

\begin{proof}[Proof of Theorem~\ref{theorem:class-VII}]
If $\ad(X)=1$, then the assertion follows from Lemma~\ref{lemma:a-1}.
If~\mbox{$\ad(X)=0$} and $X$ contains at least one curve, then the assertion follows from
Lemma~\ref{lemma:at-least-one-curve}.
Finally, if $X$ contains no curves, then the assertion follows from
Lemma~\ref{lemma:no-curves}.
\end{proof}

An alternative way to prove Theorem~\ref{theorem:class-VII}
is provided by the following more general result due to I.\,Mundet i Riera.
Our proof of Theorem~\ref{theorem:class-VII} is a simplified version of
the proof of this result given in~\cite{Riera2016}.

\begin{theorem}[{\cite[Theorem~1.1]{Riera2016}}]
\label{theorem:Riera}
Let $X$ be a compact, orientable, connected four-dimensional smooth manifold
with $\chit(X)\neq 0$. Then the group of diffeomorphisms of $X$ is Jordan.
In particular, if $X$ is a compact complex surface with non-vanishing
topological Euler characteristic, then the group $\Aut(X)$ is Jordan.
\end{theorem}

Note however that our proof of Theorem~\ref{theorem:class-VII}
implies that for a compact complex surface~$X$
with $\chit(X)\neq 0$ and containing no curves,
there exists a constant $J$ such that
for every finite subgroup $G\subset\Aut(X)$
there exists a normal \emph{cyclic} subgroup of index at most~$J$
(see Lemma~\ref{lemma:no-curves}),
while the results of \cite{Riera2016}
provide only a normal abelian subgroup
of bounded index generated by at most~$2$ elements.

\section{Hopf surfaces}
\label{section:Hopf}

In this section we study automorphism groups of Hopf surfaces, and
make some general conclusions about automorphisms of surfaces of class~VII.

Recall
that a \emph{Hopf surface} $X$ is a compact complex surface
whose universal cover is (analytically)
isomorphic to $\CC^2\setminus\{0\}$.
Thus $X\cong \left(\CC^2\setminus\{0\}\right)/\Gamma$, where $\Gamma\cong \pi_1(X)$
is a group acting freely on $\CC^2\setminus\{0\}$.
A Hopf surface $X$ is said to be \emph{primary} if~\mbox{$\pi_1(X)\cong \ZZ$}.
One can show that a primary Hopf surface is
isomorphic to a quotient
$$
X(\alpha,\beta,\lambda,n)=\left(\CC^2\setminus\{0\}\right)/\Lambda,
$$
where
$\Lambda\cong\ZZ$ is a group generated by the transformation
\begin{equation}\label{eq:Hopf}
(x,y)\mapsto (\alpha x+\lambda y^n, \beta y).
\end{equation}
Here $n$ is a positive integer, and
$\alpha$ and $\beta$ are complex numbers satisfying
$$
0 < |\alpha|\le |\beta|<1;
$$
moreover, one has $\lambda=0$, or $\alpha=\beta^n$.
A \emph{secondary} Hopf surface is a quotient of a primary Hopf surface by a
free action of a finite group. Every Hopf surface is either primary or secondary.
We refer the reader to~\mbox{\cite[\S10]{Kodaira-structure-2}} for details.

The following result shows the significance of Hopf and Inoue surfaces (see \cite{Inoue1974})
from the point of view of Enriques--Kodaira classification.

\begin{theorem}[{see \cite{Bogomolov} and~\cite{Teleman-classification}}]
\label{theorem:Bogomolov}
Every minimal surface of class VII with vanishing second Betti number is either a Hopf surface
or an Inoue surface.
\end{theorem}

Automorphisms of Hopf surfaces were studied in detail in \cite{Kato},
\cite{Kato1}, \cite{Namba}, \cite{Wehler}, and~\cite{Matumoto}.
Our approach does not use these results.

We will need the following easy observation.

\begin{lemma}\label{lemma:centralizers}
Let
$$
M=\left(
\begin{array}{cc}
\alpha & \lambda\\
0 & \beta
\end{array}
\right)\in\GL_2(\CC)
$$
be an upper triangular matrix, and $Z\subset\GL_2(\CC)$ be the centralizer
of~$M$. The following assertions hold.
\begin{enumerate}
\item
If $\alpha=\beta$ and $\lambda=0$, then $Z=\GL_2(\CC)$.
\item
If $\alpha\neq\beta$ and $\lambda=0$, then $Z\cong(\CC^*)^2$.
\item
If $\alpha=\beta$ and $\lambda\neq 0$,
then $Z\cong\CC^*\times\CC^+$.
\end{enumerate}

\end{lemma}
\begin{proof}
Simple linear algebra.
\end{proof}

\begin{lemma}\label{lemma:secondary-Hopf}
Let $X$ be a Hopf surface.
Then the group $\Aut(X)$ is Jordan.
\end{lemma}

\begin{proof}
The non-compact surface $\CC^2\setminus\{0\}$ is the universal
cover of $X$. Moreover, $X$ is obtained from~\mbox{$\CC^2\setminus\{0\}$} as a quotient by a free action of some group
$\Gamma$ that contains a subgroup $\Lambda\cong\ZZ$ of finite index such that a generator of $\Lambda$ acts
as in~\eqref{eq:Hopf}; if $X$ is a primary Hopf surface, then $\Gamma=\Lambda$, and otherwise $\Lambda$ is identified
with the fundamental group of (some) primary Hopf surface covering $X$.
By Lemma~\ref{lemma:characteristic-subgroup} we can replace $\Lambda$ by its subgroup $\Lambda_0\cong\ZZ$ such that
$\Lambda_0$ is characteristic in $\Gamma$. Since the generator of $\Lambda_0$ is a power of a generator of $\Lambda$,
it also acts on $\CC^2\setminus\{0\}$ by a transformation of type~\eqref{eq:Hopf} (but possibly with different parameters
$\alpha$, $\beta$, and $\lambda$). Therefore, without loss of generality we may assume
that $\Lambda$ itself was a characteristic subgroup of $\Gamma$.

There is an exact sequence of groups
$$
1\longrightarrow \Gamma \longrightarrow
\widetilde{\Aut}(X)\longrightarrow\Aut(X)\longrightarrow 1,
$$
where $\widetilde{\Aut}(X)$ acts by automorphisms of $\CC^2\setminus\{0\}$.
By Hartogs theorem the action of~\mbox{$\widetilde{\Aut}(X)$}
extends to~$\CC^2$ so that $\widetilde{\Aut}(X)$ fixes the origin $0\in\CC^2$.
The image of the generator of~$\Lambda$ is mapped
by the natural homomorphism
$$
\varsigma\colon\widetilde{\Aut}
(X)\longrightarrow\GL\big(T_{0,\CC^2}\big)\cong\GL_2(\CC)
$$
to the matrix
\[
M=\begin{pmatrix}
\alpha & \lambda \updelta_1^n
\\
0& \beta
\end{pmatrix}
\]
where $\updelta$ is the Kronecker symbol.

Let $G\subset\Aut(X)$ be a finite subgroup, and
$\tilde{G}$ be its preimage in $\widetilde{\Aut}(X)$.
Thus, one has~\mbox{$G\cong\tilde G/\Gamma$}.
By Corollary~\ref{corollary:fixed-point1} we know that
$\varsigma\vert_{\tilde{G}}$ is an embedding.
Let $\Omega$ be the normalizer of $\varsigma(\Lambda)$ in $\GL_2(\CC)$.
By construction~$\varsigma(\tilde G)$ is contained in the normalizer
of~$\varsigma(\Gamma)$ in $\GL_2(\CC)$, which in turn is contained in $\Omega$ because
$\Lambda$ is a characteristic subgroup of~$\Gamma$. We see that
every finite subgroup of $\Aut(X)$ is contained in the group~\mbox{$\Omega/\varsigma(\Gamma)$}.
On the other hand, $\Omega/\varsigma(\Gamma)$ is a quotient of
$\Omega/\varsigma(\Gamma)$ is a quotient of $\Omega/\varsigma(\Lambda)$
by a finite subgroup isomorphic to $\varsigma(\Gamma)/\varsigma(\Lambda)$.
Thus, by Lemma~\ref{lemma:group-theory}(ii) it is sufficient to
show that the group~\mbox{$\Omega/\varsigma(\Lambda)$} is strongly Jordan.

Since~\mbox{$\varsigma(\Lambda)\cong\ZZ$},
the group $\Omega$ has a (normal) subgroup $\Omega'$ of index at most $2$ that
coincides with the centralizer of the matrix $M$.
It remains to check that the group $\Omega'/\varsigma(\Lambda)$ is strongly Jordan.
If $\lambda=0$ and $\alpha=\beta$, then this follows from
Lemmas~\ref{lemma:centralizers}(i) and~\ref{lemma:GL-quotient-Z}.
If either $\lambda=0$ and $\alpha\neq \beta$,
or $\lambda\neq 0$ and $n\ge 2$,
then this follows from
Lemmas~\ref{lemma:centralizers}(ii) and~\ref{lemma:diagonal-quotient-Z-1}.
If $\lambda\neq 0$ and $n=1$, then this follows from
Lemmas~\ref{lemma:centralizers}(iii) and~\ref{lemma:diagonal-quotient-Z-1}.
\end{proof}

\begin{remark}
Suppose that for a primary Hopf surface $X\cong X(\alpha,\beta,\lambda,n)$ one has
$\lambda=0$ and $\alpha^k=\beta^l$ for some positive integers $k$ and~$l$.
Then there is an elliptic fibration
$$
X\cong\left(\CC^2\setminus \{0\}\right)/\Lambda \to \PP^1\cong\left(\CC^2\setminus
\{0\}\right)/\CC^*,
$$
and one has an exact sequence of groups
$$
1\longrightarrow E\longrightarrow \Aut(X)\longrightarrow
\PGL_2(\CC),
$$
where $E$ is the group of points of the elliptic
curve~$\CC^*/\ZZ$.
\end{remark}

Finally, we put together
the information about automorphisms of surfaces of class~VII.

\begin{corollary}\label{corollary:class-VII}
Let $X$ be a minimal surface of class VII. Then
the group $\Aut(X)$ is Jordan.
\end{corollary}
\begin{proof}
If the second Betti number $\bb_2(X)$ vanishes, then $X$ is either a Hopf surface or an Inoue surface by
Theorem~\ref{theorem:Bogomolov}. Thus the assertion follows from
Lemma~\ref{lemma:secondary-Hopf} and Theorem~\ref{theorem:IK} in this case.
If $\bb_2(X)$ does not vanish, then the assertion follows from
Theorem~\ref{theorem:class-VII}.
\end{proof}

\begin{remark}
Except for Hopf surfaces, there are some other types of minimal compact complex surfaces
of class~VII whose automorphism groups have been studied in detail.
For instance, this is the case for so called hyperbolic and parabolic Inoue surfaces,
see~\cite{Pinkham} and~\cite{Fujiki-ParabolicInoue},
respectively.
Note that surfaces of both of these types have positive second Betti numbers
(and thus they are not to be confused with Inoue surfaces introduced in~\cite{Inoue1974}).
Also, automorphism groups of Enoki surfaces are known due to~\mbox{\cite[Theorem~3.1]{DlousskyKohler}}
and~\mbox{\cite[Proposition~3.2(2)]{DlousskyKohler}}.
\end{remark}

\section{Kodaira surfaces}
\label{section:Kodaira}

In this section we study automorphism groups of Kodaira surfaces.

Recall (see e.g. \cite[\S\,V.5]{BHPV-2004}) that a Kodaira surface
is a compact complex surface of Kodaira dimension $0$ with odd first Betti number.
There are two types of Kodaira surfaces: primary and secondary ones.
A primary Kodaira surface
is a compact complex surface with
the following invariants \cite[Theorem~19]{Kodaira-structure-1}:
\[
\newcommand{\cquad}{\hspace{4pt}}
\KKK_X \sim 0, \cquad
\ad(X)=1,\cquad
\bb_1(X)=3,\cquad \bb_2(X)=4,\cquad \chit(X)=0,\cquad \hh^{0,1}(X)=2,\cquad
\hh^{0,2}(X)=1.
\]
A secondary Kodaira surface is a quotient of a primary Kodaira surface by a
free action of a finite cyclic group.
The invariants of a secondary Kodaira surface are \cite[\S9]{Kodaira-structure-2}:
\[
\newcommand{\cquad}{\hspace{6pt}}
\ad(X)=1,\cquad
\bb_1(X)=1,\cquad \bb_2(X)=0,\cquad \chit(X)=0,\cquad \hh^{0,1}(X)=1,\cquad
\hh^{0,2}(X)=0.
\]

Due to Theorem~\ref{theorem:IK}, we know that automorphism groups of primary Kodaira surfaces are Jordan.
Thus, we are left with the case of secondary Kodaira surfaces.

\begin{lemma}
\label{lemma:secondary-Kodaira-surface}
Let $X$ be a secondary Kodaira surface.
Then the group $\Aut(X)$ is Jordan.
\end{lemma}

\begin{proof}
Since $\ad(X)=1$, the algebraic reduction is an $\Aut(X)$-equivariant elliptic fibration~\mbox{$\pi\colon X\to B$}.
Thus there
is an exact sequence of groups
$$
1\longrightarrow
\Aut(X)_{\pi}\longrightarrow\Aut(X)\longrightarrow\Gamma\longrightarrow 1,
$$
where the action of $\Aut(X)_{\pi}$ is fiberwise with
respect to $\pi$, and $\Gamma$ is a subgroup of $\Aut(B)$.

We claim that the group $\Aut(X)_{\pi}$ is Jordan.
Indeed, if $H$ is a finite subgroup of~\mbox{$\Aut(X)_{\pi}$},
then $H$ acts faithfully on a typical fiber of~$\pi$,
which is a smooth elliptic curve.
This implies that $H$ has a normal abelian subgroup of
index at most~$6$.

Since
\[
\hh^{1,0}(X)=\bb_1(X)-\hh^{0,1}(X)= 0,
\]
the base curve $B$ is rational.
Furthermore, one has
\[
\chi(\OOO_X)=\hh^{0,0}(X)-\hh^{0,1}(X)+\hh^{0,2}(X)=1-1+0=0.
\]
By the canonical bundle formula (see e.g. \cite[Theorem~V.12.1]{BHPV-2004})
we have
\[
\KKK_X\sim\pi^*\left(\KKK_B\otimes\mathcal{L}\right)\otimes\OOO_X\left(\sum (m_i-1) F_i\right),
\]
where $F_i$ are all (reduced) multiple fibers of $\pi$, the fiber
$F_i$ has multiplicity $m_i$, and $\mathcal{L}$ is a line bundle
of degree $\chi(\OOO_X)=0$.
Since $X$ has Kodaira dimension~$0$,
we see that
\begin{equation*}
\sum (1-1/m_i)=2.
\end{equation*}
In particular, the number of multiple fibers equals either $3$ or $4$.
This means that $\Gamma$ has a finite non-empty
invariant subset in $B\cong\PP^1$ that consists of $3$ or $4$ points. Hence $\Gamma$ is finite,
so that the assertion follows by Lemma~\ref{lemma:group-theory}(i).
\end{proof}

An alternative way to prove the Jordan property for the automorphism group of a secondary Kodaira
surface is to use the fact that its canonical cover is a primary Kodaira surface
together with Lemma~\ref{lemma:group-theory}(ii) and Theorem~\ref{theorem:Riera-generators}.

\section{Non-negative Kodaira dimension}
\label{section:non-negative}

In this section we study automorphism groups of surfaces of non-negative Kodaira dimension,
and prove Theorems~\ref{theorem:Aut} and~\ref{theorem:main}.

The case of Kodaira dimension $2$ is well known.

\begin{theorem}\label{theorem:general-type}
Let $X$ be a (minimal) surface of general type. Then the group
$\Aut(X)$ is finite.
\end{theorem}
\begin{proof}
The surface $X$ is projective, see Theorem~\ref{theorem:classification}.
Thus the group $\Aut(X)$ is finite, see for instance~\cite{HMX}
where a much more general result is established for varieties of general
type of arbitrary dimension.
\end{proof}

Now we consider the case of Kodaira dimension $1$.

\begin{lemma}[{cf. \cite[Lemma~3.3]{ProkhorovShramov-dim3}}]
\label{lemma:Kodaira-dimension-1}
Let $X$ be a minimal surface of Kodaira dimension~$1$.
Then the group $\Aut(X)$
is Jordan.
\end{lemma}
\begin{proof}
Let $\phi\colon X\to B$ be the pluricanonical
fibration, where $B$ is some (smooth) curve.
It is equivariant with respect to the action of $\Aut(X)$. Thus there
is an exact sequence of groups
$$
1\longrightarrow
\Aut(X)_{\phi}\longrightarrow\Aut(X)\longrightarrow\Gamma\longrightarrow 1,
$$
where the action of $\Aut(X)_{\phi}$ is fiberwise with
respect to $\phi$, and $\Gamma$ is a subgroup of~\mbox{$\Aut(B)$}.
As in the proof of Lemma~\ref{lemma:secondary-Kodaira-surface}, we see that
the group $\Aut(X)_{\phi}$ is Jordan.
Hence by Lemma~\ref{lemma:group-theory}(i) it is enough to check that
$\Gamma$ has bounded finite subgroups.
In particular, this holds if the genus of $B$ is at least~$2$,
since the group $\Aut(B)$ is finite in this case. Thus we will assume
that the genus of $B$ is at most~$1$.

Suppose that $\phi$ has a fiber $F$ such that $F_{\red}$
is not a smooth elliptic curve. Then every irreducible component of $F$
is a rational curve, see e.g.~\mbox{\cite[\S\,V.7]{BHPV-2004}}.
Hence Lemma~\ref{lemma:rational-curve} applied to the set of irreducible components
of fibers of the morphism~$\phi$ shows
that the group~\mbox{$\Aut(X)$} is Jordan.

Therefore, we will assume that all (set-theoretic) fibers of~$\phi$ are smooth
elliptic curves. Then the topological Euler characteristic~\mbox{$\chit(X)$}
equals~$0$. By the Noether's formula one has
\[
\chi(\OOO_X)=\frac{1}{12}\left(\cc_1(X)^2+\chit(X)\right)=0.
\]
By the canonical bundle formula
we have
\[
\KKK_X\sim\phi^*\left(\KKK_B\otimes\mathcal{L}\right)\otimes\OOO_X\left(\sum (m_i-1) F_i\right),
\]
where $F_i$ are all (reduced) multiple fibers of $\phi$, the fiber
$F_i$ has multiplicity $m_i$, and $\mathcal{L}$ is a line bundle
of degree $\chi(\OOO_X)=0$.
Since $X$ has Kodaira dimension~$1$,
we see that
\begin{equation}\label{eq:mult-fibers}
2g(B)-2+\sum (1-1/m_i)=\deg \left(\KKK_B\otimes\mathcal{L}\right)+ \sum (1-1/m_i)> 0.
\end{equation}

Suppose that $B$ is an elliptic curve, so that $g(B)=1$.
Then~\eqref{eq:mult-fibers} implies that $\phi$ has at least one
multiple fiber. This means that $\Gamma$ has a finite non-empty
invariant subset in $B$, so that $\Gamma$ is finite.

Now suppose that $B$ is a rational curve, so that $g(B)=0$.
Then~\eqref{eq:mult-fibers} implies that $\phi$ has at least three
multiple fibers, cf. the proof of Lemma~\ref{lemma:secondary-Kodaira-surface}.
This means that $\Gamma$ has a finite non-empty
invariant subset in $B$ that consists of at least three points.
Therefore, $\Gamma$ is finite in this case as well.
\end{proof}

Finally, we consider the case of Kodaira dimension $0$.
The following result is well known.

\begin{theorem}
\label{theorem:torus}
Let $X=\CC^n/\Lambda$ be a complex torus. Then
\begin{equation}\label{eq:complex-torus}
\Aut(X)\cong\left( \CC^n/\Lambda\right)\rtimes \Gamma,
\end{equation}
where $\Gamma$ is isomorphic to the stabilizer of the lattice $\Lambda$ in
$\GL_n(\CC)$.
\end{theorem}
\begin{proof}
The proof is standard,
but we include it for the reader's convenience.
Let $\Gamma$ be the stabilizer of the point $0\in X$.
Then the decomposition~\eqref{eq:complex-torus} holds,
and it remains to prove that $\Gamma$ is isomorphic to the stabilizer of the lattice $\Lambda$ in
$\GL_n(\CC)$.

Since $\CC^n$ is the universal cover of~$X$,
there is an embedding $\Gamma\hookrightarrow\Aut(\CC^n)$,
and there is a point in $\Lambda$ that is invariant with respect to $\Gamma$.
We may assume that this is the origin in $\CC^n$.

Let $g$ be an element of $\Gamma$. One has $g(\Lambda)=\Lambda$.
We claim that $g\in\GL_n(\CC)$.
Indeed, let~$\lambda$ be an arbitrary element of the lattice $\Lambda$.
Consider a holomorphic map
$$
f_{\lambda}\colon\CC^n\to\CC^n,\quad f_{\lambda}(z)=g(z+\lambda)-g(z).
$$
One has $f_{\lambda}(z)\in\Lambda$ for every $z\in\CC^n$. This means that
$f_{\lambda}(z)$ is constant, so that all partial derivatives of $f_{\lambda}$
vanish. Hence the partial derivatives of $g(z)$ are periodic with respect to the lattice $\Lambda$.
This in turn means that these partial derivatives are bounded and thus constant, so that~$g(z)$ is a linear function in~$z$.
\end{proof}

\begin{remark}
A complete classification of finite groups that can act by automorphisms of a two-dimensional
complex torus (preserving a point therein) was obtained in~\cite{Fujiki}.
\end{remark}

Theorem~\ref{theorem:torus} immediately implies the following result (which is already known in a more general setup, see~\cite[Theorem~1.4]{Riera-SomeMfds}, \cite[Corollary~1.7]{Ye2017}).

\begin{corollary}
\label{corollary:torus}
Let $X$ be a complex torus. Then the group $\Aut(X)$ is Jordan.
\end{corollary}
\begin{proof}
By Lemma~\ref{lemma:group-theory}(i) it is enough to check that in the notation
of Theorem~\ref{theorem:torus} the group $\Gamma$ has bounded finite subgroups.
Since $\Gamma$ is a subgroup in the automorphism group of $\Lambda$, the latter
follows from Theorem~\ref{theorem:Minkowski}.
\end{proof}

\begin{lemma}\label{lemma:K3-Enriques}
Let $X$ be either a $K3$ surface, or an Enriques surface.
Then the group~\mbox{$\Aut(X)$} has bounded finite subgroups.
\end{lemma}
\begin{proof}
Suppose that $X$ is a $K3$ surface.
If $X$ is projective, then the assertion follows from \cite[Theorem~1.8(i)]{Prokhorov-Shramov-2013}.
If~$X$ is non-projective, then the assertion follows
from a stronger result of~\cite[Theorem~1.5]{Oguiso08}.

Now suppose that $X$ is an Enriques surface. Then it is projective (see Theorem~\ref{theorem:classification}),
so that the assertion again follows from~\cite[Theorem~1.8(i)]{Prokhorov-Shramov-2013}.
\end{proof}

Note that in the assumptions of Lemma~\ref{lemma:K3-Enriques},
the (weaker) assertion that the group~\mbox{$\Aut(X)$} is Jordan follows directly
from Theorem~\ref{theorem:class-VII} or Theorem~\ref{theorem:Riera}.

\begin{lemma}\label{lemma:bielliptic}
Let $X$ be a bielliptic surface. Then the group $\Aut(X)$ is Jordan.
\end{lemma}
\begin{proof}
The surface $X$ is projective (see Theorem~\ref{theorem:classification}).
Thus the assertion follows from
Theorem~\ref{theorem:Popov} (or~\cite{BandmanZarhin2015}, or~\cite{MengZhang},
or \cite[Theorem~1.8(ii)]{Prokhorov-Shramov-2013}).
\end{proof}

For a more precise description of automorphism groups of bielliptic surfaces,
we refer the reader to~\cite{BennettMiranda}.

\begin{corollary}\label{corollary:Kodaira-dimension-0}
Let $X$ be a minimal surface of Kodaira dimension $0$.
Then the group~\mbox{$\Aut(X)$}
is Jordan.
\end{corollary}
\begin{proof}
We know from Theorem~\ref{theorem:classification} that $X$ is either a complex torus, or a $K3$
surface, or an Enriques surface, or a bielliptic surface, or a Kodaira surface.

If $X$ is a complex torus, then the assertion holds by
Corollary~\ref{corollary:torus}.
If $X$ is a $K3$ surface or an Enriques surface, then the assertion
is implied by Lemma~\ref{lemma:K3-Enriques}.
If $X$ is a bielliptic surface, then the assertion holds
by Lemma~\ref{lemma:bielliptic}.
If $X$ is a Kodaira surface, then the assertion holds by
Theorem~\ref{theorem:IK} and Lemma~\ref{lemma:secondary-Kodaira-surface}.
\end{proof}

\begin{proposition}\label{proposition:Aut-minimal}
Let $X$ be a minimal surface. Then the group $\Aut(X)$ is Jordan.
\end{proposition}
\begin{proof}
We check the possibilities for the birational type of $X$ listed in
Theorem~\ref{theorem:classification} case by case.
If $X$ is rational or ruled,
then $X$ is projective (see Theorem~\ref{theorem:classification}),
and thus the group~\mbox{$\Aut(X)$} is Jordan by~\cite[Corollary~1.6]{Zarhin2015} or \cite{MengZhang}.
If $X$ is a surface of class~VII, then the group~\mbox{$\Aut(X)$} is Jordan by Corollary~\ref{corollary:class-VII}.
If the Kodaira dimension of~$X$ is~$0$, then
the group $\Aut(X)$ is Jordan by
Corollary~\ref{corollary:Kodaira-dimension-0}.
If the Kodaira dimension of~$X$ is~$1$, then
the group $\Aut(X)$ is Jordan by
Lemma~\ref{lemma:Kodaira-dimension-1}.
Finally, if the Kodaira dimension of~$X$ is~$2$, then
the group $\Aut(X)$ is finite by Theorem~\ref{theorem:general-type}.
\end{proof}

Now we are ready to prove Theorem~\ref{theorem:Aut}.

\begin{proof}[Proof of Theorem~\textup{\ref{theorem:Aut}}]
If $X$ is rational or ruled,
then $X$ is projective (see Theorem~\ref{theorem:classification}),
and thus the group $\Aut(X)$ is Jordan by~\cite{BandmanZarhin2015} or~\cite{MengZhang}.
Otherwise Proposition~\ref{proposition:Bir-vs-Aut} implies that there is a unique minimal
surface $X'$ birational to $X$, and
$$
\Aut(X)\subset\Bir(X)\cong\Bir(X')=\Aut(X').
$$
Now the assertion follows from Proposition~\ref{proposition:Aut-minimal}.
\end{proof}

Finally, we are going to prove Theorem~\textup{\ref{theorem:main}}.

\begin{proof}[Proof of Theorem~\textup{\ref{theorem:main}}]
There always exists a minimal surface birational to a given one, so we may assume that $X$ is a minimal surface
itself.

If $X$ is rational, then the group $\Bir(X)$ is Jordan
by Theorem~\ref{theorem:Serre}.
Also, by \cite[Theorem~4.2]{ProkhorovShramov-RC}
and Corollary~\ref{corollary:fixed-point}
every finite subgroup of $\Bir(X)$ contains a subgroup of bounded index that
can be embedded into $\GL_2(\CC)$. Hence every finite subgroup of~\mbox{$\Bir(X)$} can be generated by
a bounded number of elements.

If $X$ is ruled and non-rational, let $\phi\colon X\to B$ be the $\PP^1$-fibration over a
(smooth) curve. Since~$X$ is projective (see Theorem~\ref{theorem:classification}), the group
$\Bir(X)$ is Jordan if and only if $B$ is not an elliptic curve by
Theorem~\ref{theorem:Popov}.
Moreover, we always have an exact sequence of groups
$$
1\to\Bir(X)_\phi\to \Bir(X)\to\Aut(B),
$$
where the action of the subgroup $\Bir(X)_\phi$ is fiberwise with respect
to $\phi$. In particular, the group~\mbox{$\Bir(X)_\phi$} acts faithfully on the schematic general fiber of $\phi$,
which is a conic over the field~\mbox{$\CC(B)$}. This implies that finite subgroups
of $\Bir(X)_\phi$ are generated by a bounded number of elements. Also, finite subgroups
of $\Aut(B)$ are generated by a bounded number of elements.
Therefore, the same holds for finite subgroups
of $\Bir(X)$ as well.

In the remaining cases
we have $\Bir(X)=\Aut(X)$ by Proposition~\ref{proposition:Bir-vs-Aut}, so
the assertion follows from Proposition~\ref{proposition:Aut-minimal} and
Theorem~\ref{theorem:Riera-generators}.
\end{proof}


\def\cprime{$'$}

\end{document}